\documentclass[reqno, 12pt]{amsart}
\usepackage{array}
\usepackage{amsmath}
\usepackage{amsfonts}
\usepackage{amssymb}
\usepackage{mathrsfs}
\usepackage{enumerate}
\usepackage{amsthm}
\usepackage{verbatim}
\usepackage{amsmath, amscd}
\usepackage{caption}

\usepackage[usenames,dvipsnames]{color}
\usepackage{bm}
\usepackage{xy}
\xyoption{all}
\usepackage{color}
\usepackage{marvosym}
\usepackage{units}
\usepackage{amsbsy}
\usepackage{hyperref}
\usepackage{tikz}
\usepackage{tikz-cd}
\usepackage{marginnote}
\usetikzlibrary{matrix,arrows,backgrounds}

\newtheorem{theorem}{Theorem}[section]

\newtheorem{proposition}[theorem]{Proposition}
\newtheorem{corollary}[theorem]{Corollary}

\newtheorem*{theorem*}{Theorem}
\newtheorem*{conjecture*}{Conjecture}

\theoremstyle{definition}
\newtheorem{definition}[theorem]{Definition}

\newtheorem{remark}[theorem]{Remark}

\newtheorem{example}[theorem]{Example}

\newcommand{\op}[1]{\operatorname{#1}}

\newcommand{\newterm}{\textsf}

\newcommand{\dbcoh}[1]{\operatorname{D}^{\operatorname{b}}(\operatorname{coh }#1)}

\newcommand{\gm}{\mathbb{G}_m}

\def\Z{\op{\mathbb{Z}}}
\def\C{\op{\mathbb{C}}}
\def\R{\op{\mathbb{R}}}
\def\Q{\op{\mathbb{Q}}}
\def\F{\op{\mathcal{F}}}

\def\O{\op{\mathcal{O}}}
\def\A{\op{\mathbb{A}}}
\def\P{\op{\mathbf{P}}}

\title[Equivalences of families of stacky toric CY hypersurfaces]{Equivalences of families of stacky toric Calabi-Yau hypersurfaces}

\author[Doran]{Charles F. Doran}
\address{
  \begin{tabular}{l}
   Charles F. Doran \\
   \hspace{.1in} University of Alberta, Department of Mathematical and Statistical Sciences \\
   \hspace{.1in} Central Academic Building 632, Edmonton, AB, Canada T6G 2C7 \\
   \hspace{.1in} Email: {\bf doran@math.ualberta.ca} \\
  \end{tabular}
}

\author[Favero]{David Favero}
\address{
  \begin{tabular}{l}
   David Favero \\
   \hspace{.1in} University of Alberta, Department of Mathematical and Statistical Sciences \\
   \hspace{.1in} Central Academic Building 632, Edmonton, AB, Canada T6G 2C7 \\
         \hspace{.1in} Korean Institute for Advanced Study \\
   \hspace{.1in} 85 Hoegiro, Dongdaemun-gu, Seoul, Republic of Korea 02455 \\
   \hspace{.1in} Email: {\bf favero@ualberta.ca} \\
  \end{tabular}
}

\author[Kelly]{Tyler L. Kelly}
\address{
  \begin{tabular}{l}
   Tyler L. Kelly \\
   \hspace{.1in} University of Cambridge, Department of Pure Mathematics and Mathematical \\ \hspace{.1in} Statistics,  Wilberforce Road, Cambridge, United Kingdom CB3 0WB \\
   \hspace{.1in} Email: {\bf tlk20@dpmms.cam.ac.uk} \\
  \end{tabular}
}

\numberwithin{equation}{section}
\begin{document}

\begin{abstract}
Given the same anti-canonical linear system on two distinct toric varieties, we provide a derived equivalence between partial crepant resolutions of the corresponding stacky hypersurfaces.  The applications include: a derived unification of toric mirror constructions, calculations of Picard lattices for linear systems of quartics in $\P^3$, and a birational reduction of Reid's list to 81 families.
\end{abstract}

\maketitle
\setcounter{tocdepth}{2}
\tableofcontents

\section{Introduction}

Mirror symmetry predicts a correspondence between algebraic and symplectic geometry.  This is frequently described by looking at pairs of  Calabi-Yau varieties, $X, X^\vee$, whose complex and K\"ahler deformations are interchanged.  This summary is a bit misleading: given $X$, there is not even a single agreed upon construction in mathematics of the ``mirror'' $X^\vee$. In this article, we restrict ourselves to the toric setting where  several extremely explicit constructions have been put forward \cite{GP90, BH93, Bat, BB96, Cla08, Kra09, ACG16}.

 A priori, these differing toric mirrors look quite different. They are non-isomorphic as stacks, and are hypersurfaces in distinct toric Fano varieties.  Let us give a brief example:
\begin{example}
Consider the Calabi-Yau threefolds 
\begin{equation*}\begin{aligned}
X_1&:=\{x_1^5 + x_2^5 +x_3^5 + x_4^5 + x_5^5=0\} \subset \P^4, \\X_2&:= \{x_1^5 + x_2^4x_3 + x_3^5 + x_4^4x_5 + x_4x_5^4 = 0\}  \subset \P^4.
\end{aligned}\end{equation*}
While the Batyrev mirrors to both $X_1$ and $X_2$ both are hypersurfaces in  $\P^4/(\Z/5\Z)^3$,
their Berglund-H\"ubsch-Krawitz mirrors are
\begin{equation*}\begin{aligned}
X_1^\vee &= \{x_1^5 + x_2^5 +x_3^5 + x_4^5 + x_5^5=0\} \subset \P^4/(\Z/5\Z)^3, \\
X_2^\vee &= \{x_1^5 + x_2^4 + x_2x_4^5 + x_4^4x_5 + x_4x_5^4 = 0\} \subset \P^4(4,5,3,4,4).
\end{aligned}\end{equation*}
\end{example}
Note the Berglund-H\"ubsch-Krawitz mirrors are not even in the same toric variety and are seemingly unrelated.  However, mirror symmetry intuits that the complex deformation from $X_1$ to $X_2$ corresponds to a birational transformation between $X_1^\vee$ and $X_2^\vee$.  Furthermore, the homological mirror symmetry conjecture predicts that $X_1^\vee$ and $X_2^\vee$  should be derived equivalent, as $X_1$ and $X_2$ are symplectomorphic.   Together, these predictions can be seen as a way of unifying the constructions above through category theory and birational geometry.  With this ansatz, we will prove the following.

\begin{theorem}
 \label{EquivMirrorsHypers}
Consider two Calabi-Yau hypersurfaces $X_1$ and $X_2$ in the same toric Fano variety and suppose that one uses one of the following mirror constructions to provide mirrors $X_1^\vee$ and $ X_2^\vee$ to $X_1$ and $X_2$, respectively:
\begin{enumerate}
\item Batyrev mirror symmetry \cite{Bat},
\item Berglund-H\"ubsch-Krawitz mirror symmetry \cite{BH93, Kra09},
\item Artebani-Comparin-Guilbot mirror symmetry \cite{ACG16}, and
\item Clarke mirror symmetry \cite{Cla08}.
\end{enumerate}
Then $X_1^\vee$ is birational to $X_2^\vee$ and we have an equivalence of derived categories on the level of stacks $\dbcoh{\mathcal{X}_1^\vee} \cong \dbcoh{\mathcal{X}_2^\vee}$.
\end{theorem} 

This question of comparing toric mirror constructions via birational and derived equivalence has much recent history. In the case of Berglund-H\"ubsch-Krawitz mirror duality, this question was answered on the level of birational equivalence in \cite{Bor13, Sho14, Kel13, Cla14}. Derived equivalence of the Berglund-H\"ubsch-Krawitz mirrors to hypersurfaces in the same Gorenstein toric Fano variety was proven in \cite{FK} from an examination of the secondary fan associated to the anti-canonical linear system.  The analogous question for Batyrev-Borisov mirrors was solved for both birational equivalence \cite{Li, Cla17} and derived equivalence \cite{FK14}.

In this paper, we recover and generalize these results for hypersurfaces by showing that intersecting polytopes associated to linear systems on toric varieties leads to both birational identifications and derived equivalences among Calabi-Yau Deligne-Mumford stacks. We then specialize our results to applications for mirror constructions.

To set up our technical tool, let us introduce a bit of notation. Let $N$ and $M$ be dual lattices and $N_{\R} : = N \otimes_{\Z} \R$.   Let $\Sigma \subseteq N_{\R}$ be a complete fan and set $\nabla_{\Sigma}$ to be the convex hull of all minimal generators $u_\rho \in N$ of the rays $\rho$ in $\Sigma(1)$. Recall that for a given polytope $\nabla \subset N_{\R}$, we define its polar polytope $\nabla^\circ$ as the set 
\[
\nabla^\circ : = \{ m \in M_{\R} | \langle m, n\rangle \geq -1 \text{ for all $n \in \nabla$}\}.
\]
As usual, on a toric variety $X_{\Sigma}$, we can use the  rays $\rho \in \Sigma(1)$ to enumerate the variables $x_\rho$.  Then each lattice point $m$ in $\nabla_{\Sigma}^\circ$ corresponds to a monomial $x^m$ via the correspondence
\[
m \longleftrightarrow x^m := \prod_{\rho \in \Sigma(1)} x_\rho^{\langle m, u_{\rho}\rangle + 1}.
\]
Moreover, $m \in \nabla_{\Sigma}^\circ$ if and only if $x^m$ lies in the anti-canonical linear system for the toric variety $X_{\Sigma}$. Hence, a set of lattice points $\Xi \subseteq \nabla_{\Sigma}^\circ \cap M$ determines an anti-canonical linear system $\F_\Xi$ given by the space spanned by the corresponding monomials. Rephrasing, a linear combination of elements in $\Xi$ give a polynomial $\sum c_m x^m$ which corresponds to a hypersurface $Z_{w,\Sigma}$ in $X_{\Sigma}$ or a stacky hypersurface $\mathcal{Z}_{w, \Sigma}$ in the Cox stack $\mathcal{X}_{\Sigma}$ (see Definition~\ref{defn: cox stack}).

\begin{definition}\label{defn:good}
We say that a finite set of lattice points $\Xi \subseteq \nabla_{\Sigma}^\circ \cap M$ is \newterm{good} if 
$\op{Conv}(\op{Conv} \Xi^* \cap N)$ has a unique interior lattice point.
\end{definition}

\begin{theorem}[Theorem~\ref{thm: main theorem}]\label{introThm}
 Let $\Sigma_1, \Sigma_2\subseteq N_{\R}$ be complete fans and
    \[
\Xi \subseteq \nabla_{\Sigma_1}^\circ \cap \nabla_{\Sigma_2}^\circ \cap M
 \]
be a good collection of lattice points inside the intersection of the two corresponding anti-canonical linear systems.  Assume that the associated toric varieties $X_{\Sigma_1}, X_{\Sigma_2}$ are projective and that the polynomial corresponding to $w\in \F_\Xi$ is irreducible. Then there exist partial resolutions $\widetilde{\mathcal Z_{w, \Sigma_1}}, \widetilde{\mathcal Z_{w, \Sigma_2}}$ of the corresponding hypersurfaces $\mathcal Z_{w, \Sigma_1}, \mathcal Z_{w, \Sigma_2}$ in the toric stacks $\mathcal X_{\Sigma_1}, \mathcal X_{\Sigma_2}$ and a derived equivalence
\[
\dbcoh{\widetilde{\mathcal Z_{w, \Sigma_1}}} \cong \dbcoh{\widetilde{\mathcal Z_{w, \Sigma_2}}}.
\]
Furthermore, $Z_{w, \Sigma_1}$ and $Z_{w, \Sigma_2}$ are birational.

\end{theorem}

In Subsection~\ref{Quasismooth}, we describe when the resolutions are unnecessary and/or when the family is birational and derived equivalent to a hypersurface in a Fano Gorenstein toric variety. In Subsection~\ref{ReidsList}, we examine Reid's list of the 95 weighted-projective 3-spaces where the generic anti-canonical hypersurface has at most Gorenstein singularities. While only fourteen of the weighted-projective spaces correspond to reflexive polytopes, i.e., are Gorenstein toric Fano varieties, each family of K3 surfaces in a weighted-projective 3-space on the list corresponds to the generic family of K3 surfaces on some Gorenstein toric Fano 3-fold.  Some of these Gorenstein toric Fano 3-folds are repeated, yielding 81 distinct families after resolving singularities. In Subsection~\ref{normalforms}, we illustrate how our theorem can be used to realize many families of K3 surfaces as quartic linear systems in $\P^3$.  This recovers the normal form for a $(E_8\oplus E_7 \oplus H)$-polarized family of K3 surfaces in $\P^3$ discovered independently by Clingher and Doran \cite{CD12} and by Vinberg \cite{Vin13}. Our method, for example, provides 429 ways to realize 52 of the families on Reid's list as families of hypersurfaces in $\P^3$ up to birational equivalence. Finally, in Subsection~\ref{MirrorConstructions}, we provide a homological unification of mirror constructions, proving Theorem~\ref{EquivMirrorsHypers}.

{\bf Acknowledgements.}
 We would like to thank Andrey Novoseltsev both for enlightening mathematical discussions and technical support supplied by his expert knowledge of Sage.  We are also grateful to Paul Aspinwall, Colin Diemer, Antonella Grassi, and Mark Gross for many useful conversations and suggestions. We also thank the anonymous referee for their comments and feedback. The third-named author would also like to thank the Pacific Institute for Mathematical Sciences at the University of Alberta for its hospitality during his visits. The first author was supported by NSERC, PIMS, and a McCalla professorship at the University of Alberta. The second author was supported by  NSERC through a Discovery Grant and as a Canada Research Chair. The third author was supported in part by NSF Grant \# DMS-1401446 and EPSRC Grant EP/N004922/1.

 \section{Derived Equivalences of Hypersurfaces in Linear Systems}
 
 \subsection{Background}
We work over a field $\kappa$ of characteristic $0$. Let $N$ be a lattice. Consider a fan $\Sigma \subseteq N_{\R}$ with $n$ rays.
 We can construct a new fan in $\R^{\Sigma(1)}$ given by,
\[
\op{Cox}(\Sigma) := \{ \op{Cone}(e_\rho | \ \rho \in \sigma) | \ \sigma \in \Sigma \}.
\]
Enumerating the rays, this fan is a subfan of the standard fan for $\A^{n}_{\kappa} := \A^n$:
\[
\Sigma_t := \{ \op{Cone} (e_i | \ i \in I) | \ I \subseteq \{1, ..., n \} \}.
\]
 Hence, the associated toric variety $X_{\op{Cox}(\Sigma)}$ over $\kappa$ is an open subset of $\A^n$.

\begin{definition}
We call $X_{\op{Cox}(\Sigma)}$ the \newterm{Cox open set} associated to $\Sigma$. 
\end{definition}

This quasi-affine variety also comes with the action of a subgroup of the standard torus $\gm^n$ restricted from its scaling action on $\A^n$.   This is described as follows. Let $M$ be the dual lattice to $N$ and consider the right exact sequence given by the divisor class map,
\begin{align}
\label{eq: f}
M & \overset{\op{div}_\Sigma}{\longrightarrow} \Z^n \overset{\pi}{\longrightarrow} \op{Cl}(X_\Sigma) \to 0 \\
m & \mapsto \sum_{\rho \in \Sigma(1)} \langle u_\rho, m \rangle e_\rho  \notag
\end{align}
where $\Z^n$ represents the group of torus invariant divisors with basis $e_\rho$ and $\op{Cl}(X_\Sigma)$ is the divisor class group of $X_\Sigma$.

By applying $\op{Hom}(-, \gm)$, we get a left exact sequence
\[
0 \longrightarrow \op{Hom}(\op{Cl}(X_\Sigma), \gm) \overset{\widehat{\pi}}{\longrightarrow} \gm^{n} \overset{\widehat{\op{div}_{\Sigma}}}{\longrightarrow} \gm^{\op{dim }M}.
\]
We set 
\[
S_\Sigma := \op{Hom}(\op{Cl}(X_\Sigma), \gm) \subseteq \gm^n.
\]

\begin{definition}\label{defn: cox stack}
The \newterm{Cox stack} associated to $\Sigma$ is the global quotient stack
\[
\mathcal{X}_{\Sigma}:= [X_{\op{Cox}(\Sigma)}/S_{\Sigma} ].
\]
\end{definition}

 We can also associate a polytope to our fan
 \begin{equation}\label{defn:nabla}
 \nabla_{\Sigma} := \op{Conv}(u_\rho  | \rho \in \Sigma(1)),
 \end{equation}
where $u_\rho$ is the minimal generator for a ray $\rho \in \Sigma(1)$. Recall that any polytope $\nabla \subset N_{\R}$ has a \newterm{polar polytope} $\nabla^\circ$ in $M_{\R}$: 
\[
\nabla^\circ : = \{ m \in M_{\R} | \langle m, n \rangle \geq -1 \text{ for all $n \in \nabla$}\}.
\]
For any finite subset of lattice points
\[
\Xi := \{m_1, ..., m_t \} \subseteq \nabla_{\Sigma}^\circ \cap M,
\]
we get a linear system $\F_{\Xi}$ of anti-canonical functions on  $X_\Sigma$ defined by
  \[
  w := \sum_{m \in \Xi} c_{m} x^{m}.
  \]
Any such element  can be realized both as a hypersurface $Z_{w, \Sigma}$ in  the toric variety $X_\Sigma$ and as a hypersurface $\mathcal Z_{w, \Sigma}$ in the stack $\mathcal X_\Sigma$.  
The set of all lattice points of $\nabla_{\Sigma}^\circ$ gives the complete anti-canonical linear system $\F_{\nabla_{\Sigma}^\circ \cap M}$. 

  Let us recall the following theorem about derived equivalences amongst such hypersurfaces which comes from varying GIT quotients for the action of $S_\Sigma$ on $\A^n$.
 \begin{theorem}
Let $\Sigma_1, \Sigma_2$ be simplicial fans with full-dimensional convex support such that  
\[
\nabla_{\Sigma_1} = \nabla_{\Sigma_2}.
\]
Assume that $X_{\Sigma_1}$ and $X_{\Sigma_2}$  are quasi-projective varieties with nef anti-canonical bundles.  Then for any function $w \in \F_{\Xi}$, as above, there is an equivalence of categories,
\[
\dbcoh{\mathcal Z_{w, \Sigma_1}} \cong \dbcoh{\mathcal Z_{w, \Sigma_2}}.
\]
\label{thm: no resolution}
 \end{theorem}  
 \begin{proof}
If $\mathcal Z_{w, \Sigma_1}, \mathcal Z_{w, \Sigma_2}$ are smooth, this is a toric case of Corollary 4.5 of \cite{Kaw05}.  For toric varieties, this can be made entirely explicit using Theorem 2 and Theorem 3 of \cite{HW}.  The result also follows from Theorem 5.2.1 of \cite{BFK12} (version 2 on arXiv) or Cor 4.8 and Prop 5.5 of \cite{HL12}.  For an explicit statement of the theorem in this language, specialize Corollary 5.15(3) of \cite{FK17} to hypersurfaces.
 \end{proof}
 
 \begin{proposition}
Let $\Sigma_1, \Sigma_2 \subseteq N_{\R}$ be fans such that $\nabla_{\Sigma_1} = \nabla_{\Sigma_2}$ and 
\[
w = \sum_{m \in \nabla_{\Sigma_1}^* \cap M } c_m x^{m}.
\]
be irreducible.  Then, the varieties $Z_{w, \Sigma_1}, Z_{w, \Sigma_2}$ are birational.
 \label{prop: birational}
 \end{proposition}
 
 \begin{proof}  
By definition, these two hypersurfaces are defined by the same function on the open dense torus.
\end{proof}

\subsection{Blowups and Resolutions}
In this section, we compare partial crepant resolutions of toric hypersurfaces by reducing to the situation of Theorem~\ref{thm: no resolution} and give a proof of the main result Theorem~\ref{introThm}.  First, we give the definition of a star subdivision (called a generalized star subdivision in \cite{CLS}).
\begin{definition}
Given a fan $\Sigma \subseteq N_{\R}$ and a primitive element $v \in |\Sigma| \cap N$, we define the \newterm{star subdivision} of $\Sigma$ at $v$ to be the fan $\Sigma^*(v)$ consisting of the following cones
\begin{itemize}
\item $\sigma$, where $v \notin \sigma \in \Sigma$.

\item $\op{Cone}(\tau, v)$, where $v \notin \tau \in \Sigma$ and $\{v \} \cup \tau \subseteq \sigma \in \Sigma$.
\end{itemize}

\end{definition}

Recall that by doing a star subdivision, one obtains an induced projective toric morphism $X_{\Sigma^*(v)} \rightarrow X_{\Sigma}$ (Proposition 11.1.6 of \cite{CLS}) and that a toric resolution of singularities of $X_{\Sigma}$ can always be obtained through a sequence of star subdivisions (Theorem 11.1.9 of \cite{CLS}). We now take a certain sequence of star subdivisions which doesn't necessarily resolve $X_\Sigma$ but produces a desired structure we use later.
 
 \begin{proposition}
Let $\Sigma \subseteq N_{\R}$ be a complete fan and $ \Delta \subseteq \nabla_\Sigma^\circ$ be a full dimensional lattice polytope. Then, there is a simplicial fan $\Sigma^{\op{star}}$ with $\nabla_{\Sigma^{\op{star}}}  = \op{Conv}(\Delta^\circ \cap N)$
 obtained from a sequence of star subdivisions of $\Sigma$.

   \label{prop: resolve}
 \end{proposition}
 \begin{proof}
First take the simplicial refinement of $\Sigma$ (this is obtained by a sequence of star subdivisions, see Proposition 11.1.7 of \cite{CLS}).  This does not change the rays of $\Sigma$ hence does not alter $\nabla_\Sigma$.  Now, star subdivide the refinement in all the new lattice points in $\op{Conv}(\Delta^\circ \cap N)$ and call the resultant fan $\Sigma^{\op{star}}$. 
  \end{proof}

 \begin{theorem}
Let $\Sigma_1, \Sigma_2\subseteq N_{\R}$ be complete fans and $\Xi \subseteq \nabla_{\Sigma_1}^\circ \cap \nabla_{\Sigma_2}^\circ \cap M$
be a good collection of lattice points inside the intersection of the two corresponding anti-canonical linear systems.   Assume that the associated toric varieties $X_{\Sigma_1}, X_{\Sigma_2}$ are projective and that the polynomial corresponding to $w\in \F_\Xi$ is irreducible. Then there exist partial resolutions $\widetilde{\mathcal Z_{w, \Sigma_1}}, \widetilde{\mathcal Z_{w, \Sigma_2}}$ of the corresponding hypersurfaces $\mathcal Z_{w, \Sigma_1}, \mathcal Z_{w, \Sigma_2}$ in the toric stacks $\mathcal X_{\Sigma_1}, \mathcal X_{\Sigma_2}$ and a derived equivalence
\[
\dbcoh{\widetilde{\mathcal Z_{w, \Sigma_1}}} \cong \dbcoh{\widetilde{\mathcal Z_{w, \Sigma_2}}}.
\]
Furthermore, $Z_{w, \Sigma_1}$ and $Z_{w, \Sigma_2}$ are birational.
\label{thm: main theorem}
 \end{theorem}
 \begin{proof}
 Notice that $\nabla_{\Sigma_1}, \nabla_{\Sigma_2} \subseteq \op{Conv}(\Xi)^\circ$.  Furthermore, since they are lattice polytopes, $\nabla_{\Sigma_1}, \nabla_{\Sigma_2} \subseteq \op{Conv}(\op{Conv}(\Xi)^\circ \cap N)$.  
 
By Proposition~\ref{prop: resolve},  we may blowup $\mathcal X_{\Sigma_1}, \mathcal X_{\Sigma_2}$ to obtain $\mathcal X_{\Sigma_1^{\op{star}}}, \mathcal X_{\Sigma_2^{\op{star}}}$ with 
\[
\op{Conv}(\op{Conv}(\Xi)^\circ \cap N) = \nabla_{\Sigma_1^{\op{star}}} = \nabla_{\Sigma_2^{\op{star}}}.
\]
Since $\Xi$ is good (Definition~\ref{defn:good}), it follows that $\nabla_{\Sigma_1^{\op{star}}}, \nabla_{\Sigma_2^{\op{star}}}$ have a unique interior lattice point.  Hence, the anti-canonical divisors on $X_{\Sigma_1^{\op{star}}}, X_{\Sigma_2^{\op{star}}}$ are nef.
The statement then follows from Theorem~\ref{thm: no resolution} and Proposition~\ref{prop: birational}.
\end{proof}

\begin{remark} 
 In general, the stacks $\mathcal Z_{w, \Sigma_1}, \mathcal Z_{w, \Sigma_2}$ are not isomorphic, e.g., the inertia stacks may differ (see, e.g., Section 5.3 of \cite{Kel13}).
\end{remark}

\begin{remark}
If $\widetilde{\mathcal Z_{w, \Sigma_1}}$ and $\widetilde{\mathcal Z_{w, \Sigma_2}}$ are connected, then they have trivial canonical bundle by the adjunction formula.  Hence, the partial resolutions are crepant.
\end{remark}

\begin{corollary}\label{cor:SubsetLinSyst}
Let $\Sigma_1, \Sigma_2 \subseteq N_{\R}$  be complete simplicial fans and $
\Xi \subseteq (\nabla_{\Sigma_1}^\circ \cap \nabla_{\Sigma_2}^\circ) \cap M$
be a good subset.  Assume that $X_{\Sigma_1}, X_{\Sigma_2}$ are projective.  Suppose that $\mathcal Z_{w, \Sigma_1}, \mathcal Z_{w, \Sigma_2}$ are smooth connected hypersurfaces in $\mathcal X_{\Sigma_1}, \mathcal X_{\Sigma_2}$ defined by the same function,
\[
w = \sum_{m \in \Xi} c_m x^{m}.
\]
Then
\[
\dbcoh{\mathcal Z_{w, \Sigma_1}} \cong \dbcoh{\mathcal Z_{w, \Sigma_2}}.
\]
\end{corollary}
\begin{proof}
Using Theorem~\ref{thm: main theorem}, we know that $\dbcoh{\widetilde{\mathcal Z_{w, \Sigma_1}}} \cong \dbcoh{\widetilde{\mathcal Z_{w, \Sigma_2}}}$, where $\widetilde{\mathcal Z_{w, \Sigma_i}}$ is the zero locus in $\mathcal{X}_{\Sigma_i^{\op{star}}}$ as defined in Proposition~\ref{prop: resolve}. It suffices to prove that under these additional hypotheses that $\dbcoh{\widetilde{\mathcal Z_{w, \Sigma_i}}} \cong \dbcoh{\mathcal Z_{w, \Sigma_i}}$. Our strategy to prove this result is to compare the spaces $\tilde{\mathcal{V}} := \op{tot}(-K_{X_{\Sigma_i^{\op{star}}}})$ and $\mathcal{V} := \op{tot}(-K_{X_{\Sigma_i}})$ using geometric invariant theory. This may not be done directly, but can be done after partially compactifying $\mathcal{V}$. The desired equivalence will then be obtained by constructing said partial compactification $X_{\bar \Sigma}$ of $\mathcal{V}$ which can instead be compared with $\tilde{\mathcal{V}}$ using Corollary 4.7 of \cite{FK}.

We now argue that there exists a semiprojective toric variety $\mathcal{X}_{\bar \Sigma}$ such that $\mathcal V$ is open (torically) in $\mathcal{X}_{\bar \Sigma}$ and the support of $\bar \Sigma$ is equal to the support of the fan for $\tilde{\mathcal V}$. 
 This can be done iteratively.  Starting with the fan $\Sigma \subseteq  N_{\R} \times \R$ for $\mathcal{V}$, we can construct a new fan $\Gamma$ by taking an external star subdivision with respect to the lattice point $(n,1) \in N \times \Z$, where $n$ is a minimal generator of a ray in $\Sigma_i^{\op{star}}(1)$ but not in $\Sigma_i(1)$. Since $\Xi$ is good, the fans for $\mathcal{V}$ has full convex support, thus $\mathcal{V}$ is semiprojective.  Hence, $\Sigma$ is obtained from a regular triangulation in the hyperplane $N_{\R} \times \{1\}$ i.e.\ there exists weights $\omega_i$ for all minimal generators $(u_{\rho_i}, 1) \in N_{\R} \times \{1\}$ so that every simplex of $\Sigma$ is in $N_{\R} \times \{1\}$ is in the lower hull, as described in page 740 of \cite{CLS}. By choosing a weight $\omega$ sufficiently large for the point $(n,1)$ along with the weights $\omega_i$, one can show that the simplices defined via the star subdivision along $(n,1)$ create a regular triangulation of $\op{Conv}((n,1), (u_{\rho_i},1))$.  Since this triangulation is regular, $X_{\Gamma}$ is semiprojective.  We iterate this process until we arrive at $X_{\bar \Sigma}$.
 
 The global section $w$ corresponds to a global function $W$ on $X_{\Gamma}$ and $\tilde{\mathcal{V}}$. As $\Xi \subseteq \nabla_{\Sigma_i}^*$, we have that $\Xi \times\{1\} \subseteq |\bar\Sigma|^\vee$. For any fan $\Psi$ with support $|\bar\Sigma|$, we have a function $W$ corresponding to $w$ by taking $W = \sum_{m \in \Xi} c_m \prod_{\rho\in \Psi(1)} x_\rho^{\langle (m,1), u_\rho\rangle}$. For the line bundles $\mathcal{V}$ and $\tilde{\mathcal{V}}$, there is one new ray in their fans compared to $X_{\Sigma_i}$ and $X_{\Sigma_i^{\op{star}}}$, which we denote by $u$ corresponding to the ray generated by $(0,1)$.  For these line bundles, $W$ is simply $uw$.
 
We are now squarely in the context of Section 4 of \cite{FK}.  Namely, we have two fans: $\bar \Sigma$ and the fan $\Sigma_i^{\op{star}}$ for $\tilde{\mathcal V}$.  Both fans have ray generators in $N_{\R} \times \{1\}$, have the same support, and correspond to semiprojective toric varieties ($\tilde{\mathcal{V}}$ is also semiprojective since $\Xi$ is good).  Furthermore, $\mathcal V$ is open in $X_{\bar \Sigma}$. Hence, the result will now follow from Corollary 4.7 of \cite{FK} assuming the conditions are satisfied.  Phrased geometrically, Corollary 4.7 of loc. cit. requires that the critical locus of $W$ does not intersect the boundary $X_{\bar \Sigma} \backslash \mathcal V$.  This is shown as follows.

First, the assumption that $\mathcal Z_{w, \Sigma_i}$ is smooth implies that the critical locus $\partial W|_{\mathcal V}$ of $W$ in $\mathcal V$ is contained in the zero section of the bundle.  Hence, it is projective as $Z_{w, \Sigma_i}$ is projective.  Therefore,  $\partial W|_{\mathcal V}$ is closed in $X_{\bar \Sigma}$.  Furthermore,
\[
\partial W|_{X_{\bar \Sigma}} \cap \mathcal V = \partial W|_{\mathcal V}.
\] 
It follows that $X_{\bar \Sigma} \cap \mathcal V$ and $X_{\bar \Sigma}  \cap \partial W|_{\mathcal V}$ are open subsets which separate  $\partial W|_{X_{\bar \Sigma}}$ into two connected components.  Hence,  
$(X_{\bar \Sigma} \backslash \mathcal V) \cap \partial W|_{X_{\bar \Sigma}}$ must be empty if $ \partial w|_{X_{\bar \Sigma}}$ is connected.

We now show that $\partial W|_{X_{\bar \Sigma}}$ is connected.  First, since $Z_{w,\Sigma_i^{\op{star}}}$ is smooth and connected, the critical locus of $W$ is equal to $Z_{w,\Sigma_i^{\op{star}}}$ inside the zero section of $\tilde{\mathcal V}$. In particular, it is connected.  Now, 
the torus equivariant birational map from $
\tilde{\mathcal{V}}$
to  $X_{\bar \Sigma}$
admits a factorization (Theorem A of \cite{Wlo97})
\[
 \xymatrix{ & Z_1 \ar[rd] \ar[ld] &&  \cdots \ar[rd] \ar[ld] && Z_n \ar[rd] \ar[ld]  & \\
\tilde{\mathcal{V}} && X_1 && X_{n-1} && X_{\bar \Sigma} }
\]
where each map in the diagram is a blow-up and all maps are as varieties over $\A^1$ via the global function $W$ on each space.  Since $\partial W$ is connected on $\tilde{\mathcal{V}}$ it follows that its preimage is connected in $Z_1$.  Hence the image is connected in $X_1$  but this is just  $\partial W$ on $X_1$.  Continuing, we get that $\partial W$ is connected in $X_{\bar \Sigma}$, as desired.
\end{proof}

 \section{Applications and Examples}

\subsection{Reflexivity, genericness, and smoothness}\label{Quasismooth}

\begin{proposition}[Bertini, Mullet]
Let $X_\Sigma$ be a projective toric variety.  The generic stacky hypersurface in a basepoint free linear system $\mathcal Z \subseteq \mathcal X_{\Sigma}$ is smooth and connected.
\label{prop: Bertini}
\end{proposition}
\begin{proof}
This is Proposition 6.7 of \cite{Mul09}.
\end{proof}

\begin{proposition}
Let $\Sigma \subseteq N_{\R}$ be a simplicial fan such that $\nabla_\Sigma$ is reflexive.  Assume $X_\Sigma$ is projective.  Any linear system which contains the vertices of $\nabla_\Sigma$ is basepoint free.  
\label{prop: reflexive}
\end{proposition}

\begin{proof}
Consider the special case where $\Sigma$ is the normal fan of $\nabla^*_{\Sigma}$.  Note that any facet $F_\sigma$ of $\nabla_{\Sigma}$ is the convex hull $\op{Conv}(u_\rho \ | \ \rho \in \sigma(1))$ for some maximal cone $\sigma$. Since $\nabla_{\Sigma}$ is reflexive,   there is a lattice point $m_{\sigma}\in \nabla_{\Sigma}^*$ so that $\langle m_\sigma, f\rangle = -1$ for all $f \in F_\sigma$ hence $\langle m_\sigma, u_\rho\rangle = -1$ for all $\rho \in \sigma(1)$. By Proposition 6.1.1 of \cite{CLS}, since $m_\sigma \in \nabla_{\Sigma}^* = P_{-K_{X_{\Sigma}}}$, the anti-canonical divisor is basepoint free.  The general case follows: any simplicial fan $\Sigma'$ that satisfies our hypotheses corresponds to toric variety $X_{\Sigma'}$  which is a blow-up of the toric variety $X_{\Sigma}$ in our special case, and we then use the fact that the pullback of a basepoint free linear system is a basepoint free linear system.
\end{proof}

\begin{corollary}\label{cor: smooth reflexive}
Let $\Sigma_1, \Sigma_2 \subseteq N_{\R}$ be complete fans and $\Xi \subseteq \nabla_{\Sigma_1}^\circ \cap \nabla_{\Sigma_2}^\circ \cap M$ be a finite set.   
Assume that $X_{\Sigma_1}, X_{\Sigma_2}$  are projective and that $\op{Conv}(\Xi)$ is reflexive.
Let $\widetilde{\Sigma_{\op{Conv}(\Xi)}}$ be a simplicial resolution of the normal fan to $\op{Conv}(\Xi)$.

Then for a generic choice of coefficients
\[
w := \sum_{m \in \Xi} c_m x^{m}
\]
the following categories are equivalent
\[
\dbcoh{\mathcal Z_{w, \widetilde{\Sigma_{\op{Conv}(\Xi)}}}}  \cong \dbcoh{\widetilde{\mathcal Z_{w, \Sigma_1}}} \cong \dbcoh{\widetilde{\mathcal Z_{w, \Sigma_2}}},
\]
where $\widetilde{\mathcal Z_{w, \Sigma_i}}$ is a smooth stacky crepant resolution of the hypersurface defined by $w$ in the stack $\mathcal X_{\Sigma_i}$.  Moreover, $Z_{w, \Sigma_1},Z_{w, \Sigma_2}, Z_{w, \Sigma_{\op{Conv}(\Xi)}}$  are all birational.   
\end{corollary}

\begin{proof}
Since $\op{Conv}(\Xi)$ is reflexive, $\Xi$ is good.  Hence we may apply
Theorem~\ref{thm: main theorem} for the comparison between the $\widetilde{\mathcal Z_{w, \Sigma_i}}$.   The adjunction formula tells us that $\widetilde{\mathcal Z_{w, \Sigma_1}}, \widetilde{\mathcal Z_{w, \Sigma_2}} $ are smooth Deligne-Mumford stacks with trivial canonical bundle, hence the resolutions are crepant.  

To compare with $\mathcal Z_{w, \widetilde{\Sigma_{\op{Conv}(\Xi)}}}$, notice that it is smooth and connected by Propositions~\ref{prop: Bertini} and~\ref{prop: reflexive}.  Hence, we may apply
Corollary~\ref{cor:SubsetLinSyst}.
\end{proof}

We now will give a criteria for when the generic stacky anti-canonical hypersurface is smooth in a finite quotient of weighted projective space.

\begin{definition}
Given a subset $I\subseteq \{0,\ldots, n\}$, we say that a monomial $p \in \kappa[x_0,\ldots, x_n]$ is an \newterm{$I$-root} if $p = \prod_{i\in I} x_i^{r_i}$  for some $r_i \in \mathbb{Z}_{>0}$. We say that $p$ is an \newterm{$I$-pointer} if $p = x_j\prod_{i\in I} x_i^{r_i}$ for some $r_i \in \mathbb{Z}_{>0}$ and $j \notin I$.
\end{definition}

\begin{theorem}[Fletcher]
Let $X_\Sigma = \P^n(a_0,\ldots,a_n)/G$ where $G$ is a finite abelian group acting multiplicatively.
The generic stacky member $\mathcal Z_{w, \Sigma}$ of an anti-canonical linear system $F_\Xi$  in $X_\Sigma$ is smooth if for every nonempty $I \subseteq \{0, ..., n \}$ there exists an $I$-root or an $I$-pointer in $\Xi$.  
\label{thm: quasismooth}
\end{theorem}
\begin{proof}
Theorem 8.1 of \cite{Fl00} states the same for the generic member without a finite group action.  
The proof demonstrates that the sum over these monomials is smooth.  Since smoothness is an open property of a linear system, the result follows. 
Furthermore, since smoothness is a property of polynomials in $\A^{n+1} \backslash 0$, it does not depend on whether or not we quotient by $\gm$ or $\gm \times G$.
\end{proof}

\begin{remark}
One can also compare with similar results obtained by Kreuzer and Skarke.  See Theorem 1 of \cite{KS92}.
\end{remark}

\begin{corollary}
Let $X_{\Sigma_1}, X_{\Sigma_2}$ be quotients weighted projective spaces by a finite groups and
$\Xi \subseteq \nabla_{\Sigma_1}^\circ \cap \nabla_{\Sigma_2}^\circ \cap M$
be a good set. Assume that for every nonempty $I \subseteq \{0, ..., n \}$ there exists a $I$-root or an $I$-pointer in $\Xi$ for both $X_{\Sigma_1}, X_{\Sigma_2}$.  Then for a generic choice of coefficients
\[
w := \sum_{m \in \Xi} c_m x^{m}
\]
there is an equivalence of categories, $\dbcoh{\mathcal Z_{w, \Sigma_1}} \cong \dbcoh{\mathcal Z_{w, \Sigma_2}}.$  Moreover, $ Z_{w, \Sigma_1}, Z_{w, \Sigma_2}$ are birational.
\end{corollary}

\begin{proof}
For a generic choice of coefficients, $\mathcal Z_{w, \Sigma_1}, \mathcal Z_{w, \Sigma_2}$ are smooth by Theorem~\ref{thm: quasismooth}.  The result follows from Corollary~\ref{cor:SubsetLinSyst}.
\end{proof}

 \subsection{Equivalence in Reid's List of 95 Families}\label{ReidsList}
 
In \cite{KM12}, Kobayashi and Mase find birational correspondences between members of Reid's list of all 95 families of K3 surfaces in weighted-projective space whose generic anti-canonical hypersurface has Gorenstein singularities.  Here, we will recover these birational correspondences by considering the complete linear systems $\nabla_{\Sigma}^\circ \cap M$ for each $X_\Sigma$ on Reid's list (see Equation~\eqref{defn:nabla}).  

 \begin{corollary}\label{reidscor}
Suppose $X_{\Sigma}=\P^3(a_1,a_2,a_3,a_4)$ is a weighted-projective space on Reid's list.  The polytope $\Delta := \op{Conv}(\nabla_{\Sigma}^\circ \cap M)$ is reflexive and $Z_{w, \Sigma}$ and $Z_{w, \widetilde{\Sigma_\Delta}}$ are birational if $w \in \F_{\nabla_{\Sigma}^\circ \cap M}$ is irreducible, where $\widetilde{\Sigma_\Delta}$ is a simplicial resolution of the normal fan to $\Delta$.  Hence, generically, the minimal resolutions of $Z_{w, \Sigma}$ and $Z_{w, \widetilde{\Sigma_\Delta}}$ are isomorphic K3 surfaces.  Finally, generically,
the stacks $\mathcal{Z}_{w, \Sigma}$ and $\mathcal{Z}_{w, \widetilde{\Sigma_\Delta}}$ are smooth and there is an equivalence of categories, $
\dbcoh{\mathcal Z_{w, \Sigma}} \cong \dbcoh{\mathcal Z_{w, \widetilde{\Sigma_\Delta}}}.$ \end{corollary}
 
 \begin{proof}
 Proof of reflexivity and smoothness is exhaustive using Macaulay 2 \cite{Mac2} and Theorem~\ref{thm: quasismooth}.\footnote{The Macaulay2 code is available at  \href{http://www.ualberta.ca/~favero/code.html}{www.ualberta.ca/\~{}favero/code.html} }   
The rest follows directly from Corollary~\ref{cor: smooth reflexive}.
 \end{proof}

The weighted projective space $X_\Sigma = \P^3(a_1,a_2,a_3,a_4)$ is Gorenstein if and only if  $a_i$ divides $\sum a_i$ for all $i$.  This is also equivalent to $\nabla_\Sigma$ being reflexive, in which case $\Delta = \nabla_\Sigma^\circ$.  Only the first 14 weighted-projective spaces on Reid's list are Gorenstein.  The remaining 81 non-Gorenstein weighted-projective spaces have generic anti-canonical hypersurfaces which are  birational to generic anti-canonical hypersurfaces in a  (different) Gorenstein Fano toric variety.  After crepantly resolving, the associated stacks are derived equivalent and generically smooth.

There are several weighted-projective 3-spaces that correspond to the same $\Delta$. Using the indexing of reflexive polytopes in \cite{Sage}, we organize these cases into the following table:

\begin{center}\begin{tabular}{c|c}
Reid's Families & $\op{index}(\Delta)$ \\ \hline
14, 28, 45, 51 & 4080 \\
20, 59 & 3045 \\
26, 34 & 1114 \\
27, 49 & 1949 \\
38, 77 & 3731 
\end{tabular}
\qquad\qquad
\begin{tabular}{c|c}
Reid's Families & $\op{index}(\Delta)$ \\ \hline
43, 48 & 745 \\
46, 65, 80 & 88 \\
50, 82 & 4147 \\
56, 73 & 2 \\
68, 83, 92 & 221 \\
\end{tabular}
\end{center}
In particular, all families in the same row are pointwise birational and derived equivalent after resolving.

\begin{remark}
There are, in fact, 104 weighted projective 3-spaces such that the generic member of the complete anti-canonical linear system $\mathcal Z_{w, \Sigma}$ is smooth.  Indeed, in dimension 2 and 3 this is equivalent to $\op{Conv}(\Xi)$ being reflexive where $F_\Xi$ is the complete anti-canonical linear system. In higher dimensions this is no longer true.  An excellent summary of the intricate relationship between polytopes and generic smoothness can be found in Section 2.4 of \cite{ACG16}.
\end{remark}

\subsection{Some Normal Forms in $\P^3$}\label{normalforms}
 
Let $X_\Sigma$ be a $3$-dimensional projective toric variety.  If $\Xi \subseteq \nabla_{\Sigma}^\circ \cap M$ is a finite set and $\op{Conv}(\Xi)$ is reflexive, then, by Corollary~\ref{cor: smooth reflexive}, the generic member $Z_{w, \Sigma}$ of $\F_\Xi$ is birational to $Z_{w, \Sigma_{\op{Conv}(\Xi)}}$ where $ \Sigma_{\op{Conv}(\Xi)}$ is the normal fan to $\op{Conv}(\Xi)$.  Therefore, the resolutions are isomorphic.  In particular, the Picard lattices of the resolutions are isomorphic.

Recall that, for every weighted-projective space on Reid's list, there is a reflexive polytope $\Delta$ associated to it so that the polytope corresponds to the complete anti-canonical linear system.  Suppose that $\op{Conv}(\Xi)$ is isomorphic to one of the reflexive polytopes $\Delta$ corresponding to a weighted-projective space in Reid's list, i.e.,  $\op{Conv}(\Xi) = \op{Conv}(\nabla_{\Sigma'}^\circ \cap M)$ where $X_\Sigma'$ is a weighted projective space appearing Reid's list.  By applying Corollary~\ref{reidscor}, we see that $Z_{w, \Sigma}, Z_{w, \Sigma_{\op{Conv}(\Xi)}}$ are also birational to the corresponding hypersurface $Z_{w, \Sigma'}$ in the weighted projective space $X_{\Sigma'}$ on Reid's list.  Hence, for the resolutions,
$\widetilde{Z_{w, \Sigma}},\widetilde{Z_{w, \Sigma_{\op{Conv}(\Xi)}}},\widetilde{Z_{w, \Sigma'}}$ we can describe the Picard lattices
\[
\op{Pic}(\widetilde{Z_{w, \Sigma}}) = \op{Pic}(\widetilde{Z_{w, \Sigma_{\op{Conv}(\Xi)}}}) = \op{Pic}(\widetilde{Z_{w, \Sigma'}}),
\]
by invoking Belcastro's computation of Picard ranks for generic K3 surfaces in each weighted-projective space on Reid's list \cite{Bel02}.  

Now, consider the special case where $\Sigma$ is the normal fan to a standard $3$-dimensional simplex so that $X_\Sigma = \P^3$.  In this case, we consider linear systems $\F_\Xi$ spanned by quartic monomials in $\P^3$.  Up to taking convex hull of $\Xi \subseteq \nabla_{\Sigma}^\circ \cap M$, there are 429 such choices of linear systems of quartics so that $\op{Conv}(\Xi)$ is isomorphic to some $\Delta$ corresponding to a family on Reid's list.  These 429 choices cover exactly 52 of the 95 families in Reid's list and give 44 distinct Picard lattices.\footnote{The Macaulay2 code and full lists can be found at  \href{http://www.ualberta.ca/~favero/code.html}{www.ualberta.ca/\~{}favero/code.html} }   

In conclusion, using Corollary~\ref{cor: smooth reflexive}, we can find 429 linear systems of quartics in $\P^3$ for which we can describe the Picard lattice of a resolution of the generic member.   In total, these linear systems provide 44 distinct Picard lattices.  

\begin{remark}
Up to taking convex hull, there are 20260 linear systems of quartics in $\P^3$ such that $\op{Conv}(\Xi)$ is 3-dimensional and reflexive.  Exactly 3615 of the 4319 3-dimensional reflexive polytopes occur in such a way.\footnote{The Sage code and full lists can be found at  \href{http://www.ualberta.ca/~favero/code.html}{www.ualberta.ca/\~{}favero/code.html} }  The corresponding Picard lattice is described in \cite{Roh04} but nowhere are they listed in terms of the classification of symmetric unimodular lattices in the $K3$ lattice.
\end{remark}

\begin{example}
\label{ex: M2}
Consider the set 
\[
 \Xi = \{ (-1,-1,-1), (1,0,0),(0,1,0),(0,0,1),(0,0,0)\}
\]
and the corresponding linear system
 \[
 \F_{\mathbf{c}} = c_1 x^2yz + c_2 xy^2z + c_3 xyz^2 + c_4 w^4 + c_5 xyzw
 \]
 of quartics in $\P^3$. 
 
Then $\op{Conv}(\Xi) = \Delta_{0}$ is the reflexive polytope of index 0 in Sage.  Hence, a resolution of a generic member of this family is polarized by $(E_8 \oplus E_8 \oplus \langle -4\rangle \oplus H)$ since this appears as Reid Family 52 in the weighted-projective space $\P^3(7,8,9,12)$.
\end{example}

\begin{example}
Consider the set 
\[
\Xi = \{ (-1,-1,-1), (-1,-1,0),(0,0,0),(-1,1,0),(2,-1,0),(-1,-1,1)\}
\]
and the corresponding linear system
 \[
 \F_{\mathbf{c}} = c_1 w^4 + c_2 w^3z + c_3 xyzw + c_4y^2wz +c_5 x^3z + c_6w^2z^2
 \]
 of quartics in $\P^3$. 
 
Then $\op{Conv}(\Xi) = \Delta_{88}$ is the reflexive polytope of index 88 in Sage.  Hence, a resolution of a generic member of this family is polarized by $(E_8 \oplus E_8\oplus H)$ since this appears as Reid Families 46, 65, and 80 in the weighted-projective spaces $\P^3(5,6,22,33)$, $\P^3(3,5,11,14)$, and $\P^3(4,5,13,22)$, respectively.
\end{example}

\begin{example}
In this example we produce a family of $(E_8 \oplus E_7 \oplus H)$-polarized K3 surfaces in a blow-up of $\P^3$.  This was achieved independently by Clingher and Doran \cite{CD12} and Vinberg \cite{Vin13} using non-toric methods.  

Consider the set 
\[
\Xi = \{ (-1,-1,-1), (-1,-1,0),(0,0,0),(-1,1,0),(2,-1,0),(-1,-1,1),(0,-1,1)\}
\]
and the corresponding linear system
 \[
 \F_{\mathbf{c}} = c_1 w^4 + c_2 w^3z + c_3 xyzw + c_4y^2wz +c_5 x^3z + c_6w^2z^2 + c_7 xz^2w
 \]
 of quartics in $\P^3$.
 
Then $\op{Conv}(\Xi) = \Delta_{221}$ is the reflexive polytope of index 221 in Sage.  Hence, a resolution of a generic member of this family is polarized by $(E_8 \oplus E_7 \oplus H)$ since this appears as Reid Families 68, 83, and 92 in the weighted-projective spaces $\P^3(3,4,10,13)$, $\P^3(4,5,18,27)$, and $\P^3(3,5,11,19)$, respectively.
\end{example}

\begin{remark}
Bruzzo and Grassi (Theorem 3.8 of \cite{BG12}) prove that, given the family $\F$ of anti-canonical hypersurfaces in a complete, simplicial toric variety $X_\Sigma$, the Picard rank of a very general hypersurface $X \in \F$ is the Picard rank of  $X_{\Sigma}$. It is an interesting problem to find subloci of hypersurfaces of higher Picard rank than $X_{\Sigma}$. While our resolutions are of higher Picard rank, the hypersurfaces themselves are not, a priori, in the Noether-Lefschetz locus as they are sometimes not smooth.
\end{remark}

\subsection{Equivalences of Exotic Mirror Constructions}\label{MirrorConstructions}

\subsubsection{Artebani-Comparin-Guilbot Mirror Symmetry}
We now prove Theorem~\ref{EquivMirrorsHypers}.
The following definition comes form \cite{ACG16}.
\begin{definition}
Let $\Delta_1 \subseteq \Delta_2$ be polytopes in $M_{\Q}$.  We say that $(\Delta_1, \Delta_2)$ is a \newterm{good pair} if $\Delta_1$ and $\Delta_2^*$ are lattice polytopes with a unique interior lattice point. 
\end{definition}

If $(\Delta_1, \Delta_2)$ is a good pair, then we can define a family of Calabi-Yau hypersurfaces by looking at the special linear system associated to $\Xi = \Delta_1 \cap M$ corresponding to anticanonical sections in the toric variety $X_{\Sigma_{\Delta_2}}$. Note that $\nabla_{\Sigma_{\Delta_2}} = \Delta_2^*$ by the construction of the anticanonical polytope $\Delta_2$.  Notice that if $(\Delta_1, \Delta_2)$ is a good pair, then so is $(\Delta_2^*, \Delta_1^*)$. In this situation, we may regard the lattice points of $\Delta_2^*$ as a linear system $\Xi^*$ on $X_{\Sigma_{\Delta_1^*}}$ where $\nabla_{\Sigma_{\Delta_1^*}} = \Delta_1$. Artebani, Comparin, and Guilbot use this to propose the following mirror duality
\[
Z_{w, \Sigma_{\Delta_2}} \longleftrightarrow  Z_{\hat w, \Sigma_{\Delta_1^*}} 
\]
where $w$ is a generic member of the linear system $\op{Conv}(\Delta_2^*) \cap N$ and $\hat w$ is a generic member of the linear system $\op{Conv}(\Delta_1) \cap M$. 

Fixing $\Delta_2$ and choosing two good pairs $(\Delta_1, \Delta_2)$ and $(\Delta_1', \Delta_2)$, amounts to being able to choose different hypersurfaces in $X_{\Sigma_{\Delta_2}}$ which are symplectomorphic.   Hence, mirror symmetry predicts that the $B$-model of the mirror remains the same.  This phenomenon is realized via derived categories.

 \begin{corollary}\label{ACGthm}
 Let $(\Delta_1', \Delta_2)$ and $(\Delta_1, \Delta_2)$ be two good pairs. Then for a generic choice of 
 \[
\widehat w := \sum_{m \in \op{Conv}(\Delta_2^*) \cap M} c_m x^m,
 \]
 the two corresponding mirror coarse moduli spaces
$Z_{\widehat w, \Sigma_{\Delta_1}}, Z_{\widehat w, \Sigma_{\Delta_1'}}$
 are birational and the partial resolutions $\widetilde{\mathcal{Z}_{\widehat w, \Sigma_{\Delta_1}}}$ and $\widetilde{\mathcal{Z}_{\widehat w, \Sigma_{\Delta_1'}}}$ are derived equivalent. Moreover, if both ${\mathcal{Z}_{\widehat w, \Sigma_{\Delta_1}}}$ and ${\mathcal{Z}_{\widehat w, \Sigma_{\Delta_1'}}}$ are smooth, then they are derived equivalent.
 \end{corollary}
 \begin{proof}
 
 Set $\Xi = \op{Conv}(\Delta_2^*) \cap M$.
Since $\Delta_2^*$ has a unique interior lattice point, $\Xi$ is good. The polytopes $\Delta_1$ and $\Delta_1'$ have a unique interior point and $\nabla_{\Sigma_{\Delta_1}} = \Delta_1$ and $\nabla_{\Sigma_{\Delta_1'}} = \Delta'_1$, hence $X_{\Sigma_{\Delta_1}}$ and $X_{\Sigma_{\Delta_1'}}$ have nef anticanonical bundles. The result now follows from Theorem~\ref{thm: main theorem} and Corollary~\ref{cor:SubsetLinSyst}, depending on the smoothness of the stacks.
 \end{proof}
 
 \begin{remark}
 The birational aspect of this result was first proven as Proposition 3.5 in \cite{ACG16}.  
 \end{remark}

 \subsubsection{Clarke Mirror Symmetry}
 While Clarke's construction was originally created in the Landau-Ginzburg setting, we will specialize to the Calabi-Yau case and give an interpretation along the lines of \cite{ACG16}. In \cite{Cla08}, Clarke proves his construction is a generalization of the Batyrev and Berglund-H\"ubsch-Krawitz constructions. As always, $M$ and $N$ are dual lattices and $\Sigma \subseteq N_{\R}$ a fan.  We work over the field $\kappa = \C$.
Recall from Equation~\eqref{eq: f} that there is a right exact sequence:
 \begin{align}
 \label{eq: div}
M & \overset{\op{div}_\Sigma}{\longrightarrow} \Z^n \overset{\pi}{\longrightarrow} \op{Cl}(X_\Sigma) \to 0 \\
m & \mapsto \sum_{\rho \in \Sigma(1)} \langle u_\rho, m \rangle e_\rho.  \notag
\end{align}
Now consider a finite set $\Xi \subseteq \nabla_\Sigma^\circ \cap M$ with the origin in the interior of its convex hull. This gives another right exact sequence,
 \begin{align}
 \label{eq: monmap}
N & \overset{\op{mon}_\Xi}{\longrightarrow} \Z^\Xi \overset{q}{\longrightarrow} \op{coker}(\op{mon}_\Xi) \to 0 \\
n & \mapsto \sum_{m \in \Xi} \langle m, n \rangle e_m.  \notag
\end{align}

The realization of $X_\Sigma$ as a GIT quotient amounts to a choice of a $S_\Sigma$-linearization of $\O_{\mathbb A^n}$ or, equivalently, an element of  $D \in \op{Cl}(X_\Sigma)$. Indeed, the fan $\Sigma$ can be recovered from \eqref{eq: div} and $D$ as the normal fan to the polytope $(\pi \otimes_{\Z} \R)^{-1}(D \otimes_{\Z} \R) \cap \R^n_{\geq 0}$.

Similarly, a specific hypersurface with nonzero monomial coefficients in $\F_\Xi$ is a choice of coefficients $\mathbf{c} \in (\C^*)^\Xi$ which, up to reparametrization by the torus $N \otimes_{\Z} (\C^*)^\Xi$, amounts to a choice $q(\mathbf{c}) \in  \op{coker}(\op{mon}_\Xi) \otimes_{\Z} \C$.  Here, the $\Z$-module structure on $\C$ is the multiplicative structure.  

 Clarke's mirror construction simply exchanges the roles of the two right exact sequences above and the choices $D \otimes_{\Z} \C^*,  q(\mathbf{c})$. More precisely, given $X_\Sigma$ and an equivalence class $q(\mathbf{c})$ in a linear system $\Xi$, we get two stacks,
 \[
 \mathcal Z_{q(\mathbf{c}), \Sigma}, \mathcal Z_{D \otimes_{\Z} \C^*, \Sigma_{\Xi, q(\mathbf{c})}},
 \]
 where 
$\Sigma_{\Xi, q(\mathbf{c})}$ 
is the fan associated to \eqref{eq: monmap} and $q(\mathbf{c})$, i.e., the normal fan to the polytope $(q \otimes_{\Z} \R)^{-1}(\op{im} q(\mathbf{c})) \cap \R^\Xi_{\geq 0}$.
\begin{definition}[Clarke \cite{Cla08}]
The stacks $ \mathcal Z_{q(\mathbf{c}), \Sigma}$ and $ \mathcal Z_{D \otimes_{\Z} \C^*, \Sigma_{\Xi, q(\mathbf{c})}}$ are called \newterm{mirror} to one another.
 \end{definition}
 
 \begin{remark}
To complete the symmetry, one can more generally choose the right exact sequence \eqref{eq: div} and an element $\bar{D} \in \op{Cl}(X_\Sigma) \otimes_{\Z} \C$.  This amounts to a choice of $B$-field on $X_\Sigma$ \cite{Cla08}.
 \end{remark}
 
 \begin{remark}
Since we only choose an equivalence class of the defining polynomial $w$, the stacks above are only defined as subsets of $\mathcal X_\Sigma, \mathcal X_{\Sigma_{\Xi}}$ up to reparametrization by $\gm^{\Sigma(1)}, \gm^\Xi$.
 \end{remark}

 If we choose different anti-canonical linear systems $\Xi_1, \Xi_2$ and nonzero coefficients 
\[
\mathbf{c}_i : \Xi_i \to \C^*,
\]
then we get two smooth stacky Calabi-Yau hypersurfaces  $\mathcal Z_{q(\mathbf{c}_1), \Sigma}$ and $\mathcal Z_{q(\mathbf{c}_2), \Sigma}$ in $\mathcal X_\Sigma$.  While this varies only the complex structure of the Calabi-Yau hypersurface, the mirrors will lie in different toric varieties, which gives us the same phenomenon as before. We show this with a few extra assumptions on the toric variety $X_{\Sigma}$ and linear systems $\Xi_i$:
 \begin{corollary}
 \label{cor: Clarke}
Let $X_{\Sigma}$ be a projective variety, $\nabla_{\Sigma}\cap N$ a good set, and $\Xi_1, \Xi_2 \subseteq \nabla_{\Sigma}^\circ$
 finite sets such that $\op{Conv}(\Xi_i)$ is full dimensional with the origin in its interior.  
Choose coefficients $\mathbf{c}_i : \Xi_i \to \C^*$ and consider the stacky hypersurfaces  $\mathcal Z_{q(\mathbf{c}_1), \Sigma}, \mathcal Z_{q(\mathbf{c}_2), \Sigma} \subseteq \mathcal X_\Sigma$. Then the two corresponding mirror coarse moduli spaces $Z_{D \otimes_{\Z} \C^*, \Sigma_{\Xi_1, q(\mathbf{c}_1)}}, Z_{D \otimes_{\Z} \C^*, \Sigma_{\Xi_2, q(\mathbf{c}_2)}}$
 are birational. Moreover, the partial resolutions $\widetilde{\mathcal Z_{D \otimes_{\Z} \C^*, \Sigma_{\Xi_1, q(\mathbf{c}_1)}}}$ and $\widetilde{\mathcal Z_{D \otimes_{\Z} \C^*, \Sigma_{\Xi_2, q(\mathbf{c}_2)}}}$ are   derived equivalent.
\end{corollary}

\begin{proof}
The mirror stack $\mathcal Z_{D \otimes_{\Z} \C^*, \Sigma_{\Xi_i, q(\mathbf{c}_i)}}$
is an anti-canonical member of $F_{\nabla_{\Sigma} \cap N}$ in $X_{\Sigma_{\Xi_i, q(\mathbf{c}_i)}}$.  
By looking at the exact sequence~\eqref{eq: monmap}, we see that the rays of 
$\nabla_{\Sigma_{\Xi_i, q(\mathbf{c}_i)}}$ are the elements of $\Xi_i$.
Therefore, $\nabla_{{\Sigma_{\Xi_i, q(\mathbf{c}_i)}}} = \op{Conv}(\Xi_i) \subseteq M_{\R}$
Hence,
\[
\nabla_{\Sigma} \subseteq \nabla_{{\Sigma_{\Xi_1, q(\mathbf{c}_1)}}}^\circ \cap \nabla_{{\Sigma_{\Xi_2, q(\mathbf{c}_2)}}}^\circ \subseteq N_{\R} .
\]
Notice that $X_{ \Sigma_{\Xi_i, q(\mathbf{c}_i)}}$ has a nef anticanonical bundle since $\Xi_i \subseteq \op{Conv}(\op{Conv}(\nabla_{\Sigma})^* \cap M)$ has a unique interior lattice point. 
Therefore, $\mathcal Z_{D \otimes_{\Z} \C^*, \Sigma_{\Xi_1, q(\mathbf{c}_1)}}, \mathcal Z_{D \otimes_{\Z} \C^*, \Sigma_{\Xi_2, q(\mathbf{c}_2)}}$ satisfy the hypotheses of Theorem~\ref{thm: main theorem}.
\end{proof}

\begin{proof}[Proof of Theorem~\ref{EquivMirrorsHypers}]
The mirror construction of Artebani-Comparin-Guilbot is a special case of Clarke's construction where $ (\op{Conv}(\Xi_1), \nabla_{\Sigma}^*) = (\Delta_1, \Delta_2) $ is a good pair, so the theorem follows from Corollary~\ref{cor: Clarke}. Batyrev mirror symmetry is the special case of Corollary~\ref{ACGthm} where $\Delta_1 = \Delta_2$.  Similarly, Berglund-H\"ubsch-Krawitz mirror symmetry is a special case by Theorem 2 of \cite{ACG16} and does not require resolutions as BHK mirrors are smooth. 
\end{proof}
 
 \begin{remark}
In the special case of Berglund-H\"ubsch-Krawitz mirror symmetry, the birational aspect was first proven by Shoemaker in \cite{Sho14} and the derived aspect was proven in \cite{FK}.  The latter assumes that $X_\Sigma$ is Gorenstein.  The result above drops the Gorenstein assumption. 
\end{remark}

\begin{remark}
Aspinwall and Plesser \cite{AP} provide a general mirror construction for Landau-Ginzburg models associated to two Gorenstein cones which are contained in one another's duals. In their theory, they allow non-geometric phases as gauged linear sigma models. Comparing these phases relates to an algebraic condition which inspired our results, see Section 4 and especially Corollary 4.7 of \cite{FK}. The method of the current paper is a toric description which focuses only on geometric phases related to Gorenstein cones of index 1.
\end{remark}

 \end{document}